\documentclass[a4paper,11pt]{article}
\usepackage{pdfpages}
\pdfoutput=1
\usepackage{amsmath,amsfonts,amssymb,amsthm,graphicx,url}
\usepackage{enumerate}
\usepackage{caption}
\usepackage{fancyvrb}
\usepackage{booktabs}
\usepackage{fullpage}
\usepackage{float}
\usepackage{color}
\usepackage{pinlabel}
\newtheorem{prop}{Proposition}
\newtheorem{lemma}[prop]{Lemma}
\DeclareMathOperator{\sign}{sign}
\DeclareMathOperator{\Col}{Col}

\newenvironment{hproof}{%
  \proof}{\endproof}
\title{Unlinking Numbers of Links with Crossing Number 10}
\author{Lavinia Bulai}
\date{}
\begin{document}
\maketitle
\begin{abstract}
In this paper we investigate the unlinking numbers of 10-crossing links. We make use of various link invariants and explore their behaviour when crossings are changed. The methods we describe have been used previously to compute unlinking numbers of links with crossing number at most 9. Ultimately, we find the unlinking numbers of all but 2 of the 287 prime, non-split links with crossing number 10.
\end{abstract}
\section{Introduction}
A \textit{knot} can be thought of as a knotted piece of string with cross-section a single point and ends glued together to form a closed curve. A \textit{link} is a collection of knots, each knot representing a \textit{component} of the link. A \textit{sublink} of a link is the disjoint union of some of its components. Formally, a knot is a smooth isotopy class of embeddings of $S^1$ in $\mathbb{R}^3$ or $S^3$. Similarly, a \textit{link} is a smooth isotopy class of embeddings of a disjoint union of one or more circles in $\mathbb{R}^3$ or $S^3$. A \textit{smooth isotopy} is a smooth map $F:S^1\sqcup\dots\sqcup S^1\times[0,1]\rightarrow\mathbb{R}^3$ together with a family of embeddings $f_t:S^1\sqcup\dots\sqcup S^1\rightarrow\mathbb{R}^3$, such that $f_t(x)=F(x,t)$ for all $x\in S^1\sqcup\dots\sqcup S^1$ and $t\in[0,1]$. A link is \textit{trivial} if it is isotopic to the disjoint union of finitely many circles in a plane.

\begin{figure}[b]
\centering
\includegraphics[scale=0.6]{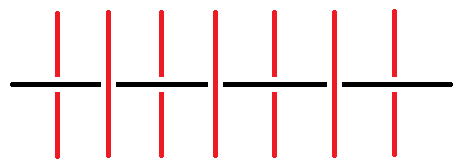}
\caption{The black strand goes alternately over and under the red one}
\label{fig:alter}
\end{figure}
A link is \textit{oriented} if each of its components is assigned an orientation. There are $2^n$ ways to orient a link with $n$ components, by adding an arrow on each knot, pointing in one of two possible directions. A projection of a link onto a plane together with a set of instructions on under-crossings and over-crossings that suffice to reconstruct the original link is referred to as a \textit{link diagram}. We assume the projection is injective, except for some double points. If the crossings are such that one goes under and over alternately when traveling along each component from an arbitrary point back to itself, then the link diagram is said to be \textit{alternating}. This property is illustrated in Figure~\ref{fig:alter}. A link is \textit{alternating} if it admits an alternating diagram. A \textit{split link} is a link that has a projection as a disconnected diagram. Otherwise, if every diagram of the link is connected, the link is said to be \textit{non-split}. If every diagram of a link is such that any line intersecting the diagram in two points divides the link into two subsets, one of them isotopic to an embedded line segment via an isotopy fixing the two endpoints, then the link is said to be \textit{prime}.

The \textit{crossing number} of a link is the minimal number of crossings in any of its diagrams. The operation of swapping the two strands that form a crossing, such that the under-crossing becomes the over-crossing and vice-versa, is known as \textit{changing a crossing}. With a sensible choice of crossing changes, one can obtain the trivial link from any given diagram. The \textit{unlinking number} is the minimal number of crossings one has to change in order to obtain the trivial link, where the minimum is taken over all diagrams of the link. In general, unlinking numbers are difficult to determine. In this paper we investigate the unlinking number of each of the $287$ prime, non-split links with crossing number $10$ and at least $2$ components, by finding constraints on the values it can take. Methods developed by Borodzik-Friedl-Powell \cite{boro}, Kauffman-Taylor \cite{kauffman}, Kawauchi \cite{kawa}, Kohn \cite{kohn}, Murasugi \cite{mura} and Nagel-Owens \cite{owens} give us lower bounds, whereas upper bounds follow from experiment. Of the links we looked at, the unlinking numbers of $2$ are still unknown and require new techniques to be developed. Good references for basics of knot theory are Adams \cite{adams}, Cromwell \cite{cromwell}, Lickorish \cite{ray} and Livingston \cite{charles}.

In Section \ref{bound} we describe various techniques that can be used to produce lower bounds on unlinking numbers. In Section \ref{table} we give a table of the $10$-crossing links and their unlinking numbers, with the exception of two links. For each of these links, we indicate in the table the technique with which the claimed lower bound is produced.

\textbf{Acknowledgements:} I am grateful to Dr Brendan Owens for supporting and encouraging me to write this paper; to the London Mathematical Society for funding my research; to Matthias Nagel and Mark Powell for their useful comments and feedback.   
\section{Lower bounds on unlinking numbers}\label{bound}
All the methods we will use throughout this paper to compute unlinking numbers of links with crossing number $10$ have previously been used to find unlinking numbers of links with crossing number $9$ or less. 

We begin with a lemma about real symmetric matrices. The \textit{signature} $\sign A$ of a real symmetric matrix $A$ is the number of positive eigenvalues minus the number of negative eigenvalues, counted with multiplicities. The \textit{nullity} of a matrix is the dimension of its kernel.
\begin{lemma}\label{prima}
Let $A$ be an $n\times n$ real symmetric matrix. Suppose that the matrix $B$ is identical to $A$, apart from one diagonal entry, say $b_{ii}\neq a_{ii}$ where $b_{ii}\in\mathbb{R}$, for some $i\in\left\{1,\dots,n\right\}$. It follows that: 
\begin{enumerate}[i)]
\item the nullity of $B$ differs from the nullity of $A$ by at most 1.
\item if $A$ and $B$ have the same nullity and $b_{ii}>a_{ii}$, then the signature of $B$ and the signature of $A$ are related by either $\sign B=\sign A$ or $\sign B=\sign A+2$.
\item if $A$ and $B$ have different nullities and $b_{ii}>a_{ii}$, then $\sign B=\sign A+1$. 
\end{enumerate}
\end{lemma}
\begin{hproof}
$i)$ The rank of the matrix $A$ is the dimension of its column space, which in turn is equal to the number of linearly independent columns. By changing the diagonal entry $a_{ii}$ for some $i\in\left\{1,\dots,n\right\}$, the column $i$ will also change, hence the rank of $A$ increases by one, stays the same, or decreases by one. However, the change has no effect on the size of $A$. From the Rank-Nullity Theorem it follows that, as the rank changes, the nullity of $A$ will either decrease by one, stay the same or increase by one.\\
$ii)$, $iii)$ These statements can be proved by considering the sequence of leading principal minors of the matrix $A$, as in the proof of Theorem 4 in \cite{jones}.
\end{hproof}
\subsection{Linking number}
Let $D$ be a diagram of the oriented link $L$, and $c$ a crossing. There are two possible configurations near $c$, as illustrated in Figure \ref{fig:orient}. The crossing on the left is said to be \textit{positive}, whereas the crossing on the right is \textit{negative}. Let
\begin{equation*}
\epsilon(c) =
\left\{
	\begin{array}{ll}
		1  & \mbox{if $c$ is a positive crossing,} \\
		-1 & \mbox{if $c$ is a negative crossing,} 
\end{array}
\right.
\end{equation*}
and let $L_1$ and $L_2$ be disjoint sublinks of $L$, such that $L=L_1\sqcup L_2$. In the diagram of $L$, a crossing may be classified according to the origin of the two strands that form it: $L_1$ with itself, $L_2$ with itself, or $L_1$ with $L_2$. The \textit{linking number} of $L_1$ and $L_2$ is defined as
\begin{equation*}
lk_D(L_1,L_2)=\frac{1}{2}\sum_{c\in L_1\cap L_2}\epsilon(c),
\end{equation*}
where we write $c\in L_1\cap L_2$ if one of the strands in the crossing belongs to $L_1$, and the other to $L_2$. Once an orientation is fixed, the linking number does not depend on the choice of diagram, so we can refer to it as $lk(L_1,L_2)$. Thus the linking number is an invariant of the link and the chosen sublinks, and a measure of the number of times one sublink winds around the other.
\begin{figure}[h]
\labellist
\small\hair 2pt
\pinlabel positive at 85 0
\pinlabel negative at 277 0
\endlabellist
\centering
\includegraphics[scale=0.5]{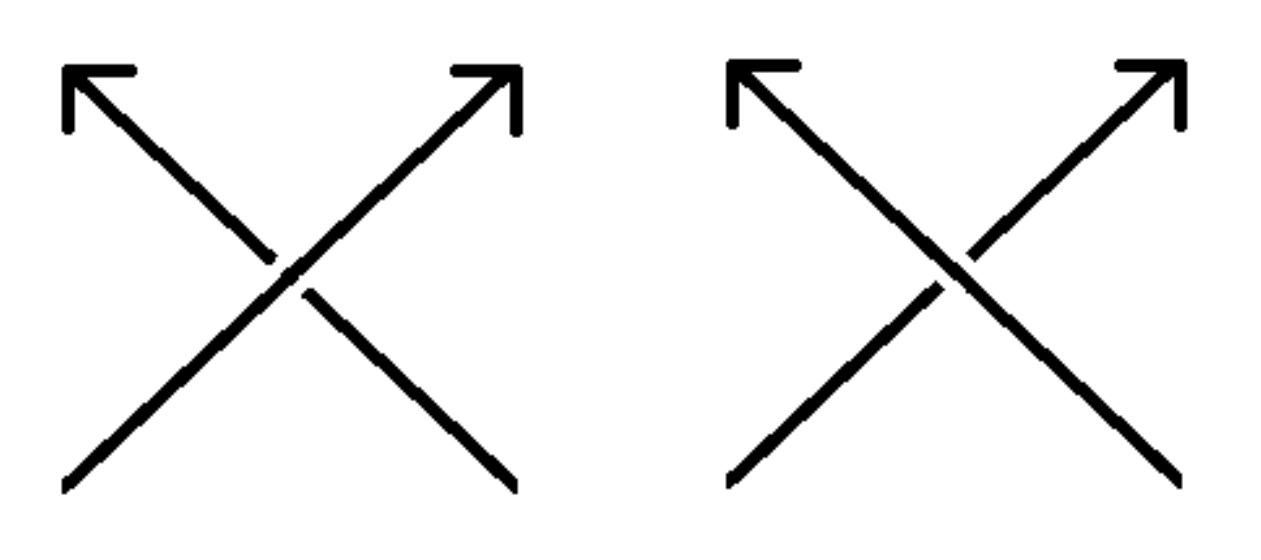}
\caption{Crossing type in an oriented link}
\label{fig:orient}
\end{figure}
\begin{prop}\cite[Theorem 1]{kohn}\label{primpropo}
Let $L=L_1\sqcup L_2$ be an oriented link in $\mathbb{R}^3$, where $L_1$ and $L_2$ are disjoint sublinks of $L$. Then the unlinking number of $L$ satisfies
\begin{equation*}
u(L) \geq u(L_1)+u(L_2)+|lk(L_1,L_2)|,
\end{equation*}
where $lk(L_1,L_2)$ is the linking number of $L_1$ and $L_2$. 
\end{prop}
\begin{proof}
Consider some crossing in a diagram $D$ of the link $L$. If both strands belong to the sublink $L_1$, then changing the crossing will have no effect on the underlying structure of the sublink $L_2$ or on the linking number of $L_1$ and $L_2$. Similarly, if both strands belongs to $L_2$, then changing the crossing will not affect $L_1$ or the linking number of the sublinks. However, if one strand belongs to $L_1$ and the other to $L_2$, then changing the crossing will have no effect on the two sublinks, but the linking number will change by one. Let us now consider an unlinking sequence that realises $u(L)$. The number of crossing changes between $L_1$ and $L_2$ is then bounded below by $|lk(L_1,L_2)|$, and the number of crossing changes completely in $L_1$ or completely in $L_2$ is bounded below by $u(L_1)$ and $u(L_2)$ respectively, thus proving the inequality.
\end{proof}
To illustrate the application of this method, consider the link $L10n96$, oriented as in Figure~\ref{fig:nomad}. Let the sublinks $L_1$ and $L_2$ both be Hopf links -- red with blue, and green with purple, respectively. The linking number of $L_1$ and $L_2$ is $3$, and it follows from an easy application of Proposition \ref{primpropo} that the unlinking number of a Hopf link is $1$, so that $u(L10n96) \geq 5$. Therefore, the link has unlinking number $5$, as it can be converted to the trivial link with $4$ components by changing the $5$ crossings indicated in the figure.
\begin{figure}[h]
\centering
\includegraphics[scale=0.3]{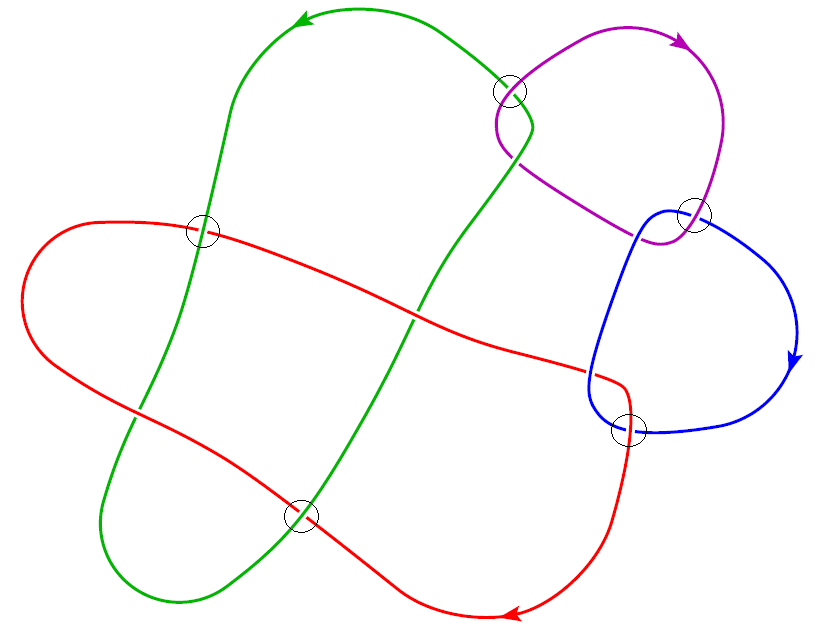}
\caption{One possible way to unlink $L10n96$}
\label{fig:nomad}
\end{figure}
\subsection{Link signature}\label{signa}
For the next method, let us begin by describing a formula for the signature of a link. Consider a diagram of the link $L$ with chessboard shading, so that no two adjacent regions share the same colour. Assign an \textit{incidence number} $\iota(c)$ to each crossing in the diagram, by letting 
\begin{equation*}
\iota(c) =
\left\{
	\begin{array}{ll}
		1  & \mbox{if $c$ is a right-handed crossing,} \\
		-1 & \mbox{if $c$ is a left-handed crossing.} 
\end{array}
\right.
\end{equation*}
Handedness is illustrated in Figure \ref{fig:chess}. Note that this is defined using the shading, and is independent of orientation. Let the $n+1$ unshaded regions in the diagram of $L$ be $R_0,R_1,\dots,R_n$. Construct the square matrix $G'=(g_{ij})$, with entries
\begin{equation*}
g_{ij} =
\left\{
	\begin{array}{ll}
		-\sum \iota(c)  & \mbox{if $i\neq j$, summing over crossings $c$ incident to both $R_i$ and $R_j$,} \\
\\
		-\sum\limits_{k=0, k\neq i}^{k=n}g_{ik} & \mbox{if $i=j$.} 
\end{array}
\right.
\end{equation*}
\begin{figure}[t]
\centering
\labellist
\small\hair 2pt
\pinlabel right-handed at 180 0
\pinlabel left-handed at 625 0
\endlabellist
\includegraphics[scale=0.2]{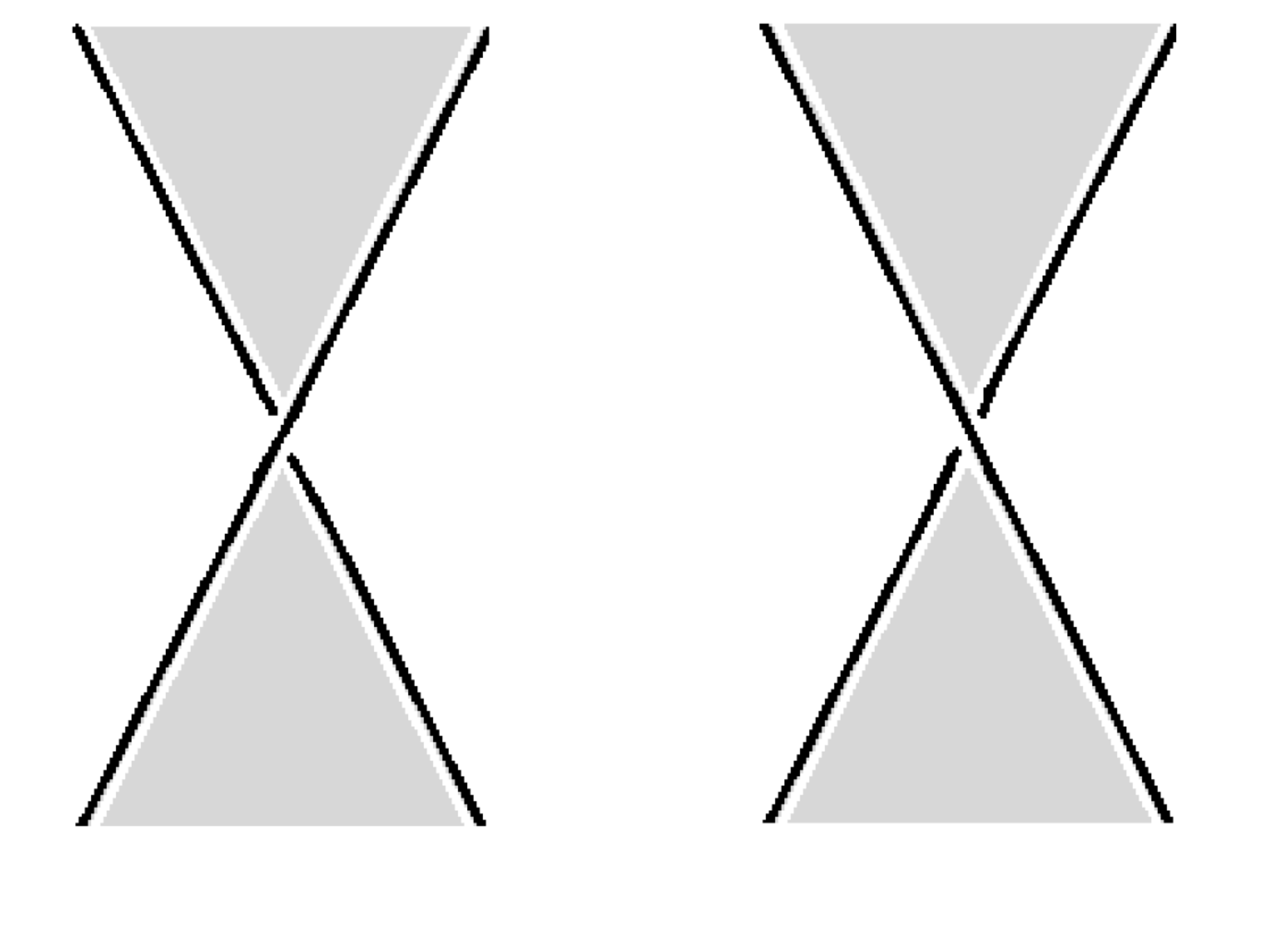}
\caption{Crossings in a chessboard-shaded diagram}
\label{fig:chess}
\end{figure}

After deleting the $0^{th}$ row and column of $G'$, another matrix is obtained, namely the symmetric square integer \textit{Goeritz matrix} $G$ of the chessboard-shaded link diagram. Let us now orient the link and consider a crossing $c$ in its diagram. If we discard information on under-crossing and over-crossing, then there are two possible configurations near $c$, type I and type II, as illustrated in Figure \ref{fig:over}. Define 
\begin{equation*}
\mu=\sum_{\tiny\mbox{type II}}\iota(c), 
\end{equation*}
where the sum is taken over all crossings of type II in the diagram of the link. Then the \textit{signature} of the link is given by
\begin{equation*}\label{ostea}
\sigma(L)=\sign G-\mu,\tag{*}
\end{equation*}
where $\sign G$ is the signature of the Goeritz matrix of the diagram. This definition of signature is due to Gordon-Litherland \cite{litherland}, who proved it to be equivalent to an older definition using Seifert surfaces. Signature is a link invariant --- once an orientation is fixed, the signature remains constant under isotopy. This was proved in \cite{trot} for knots and in \cite{mura} for links. 
\begin{figure}[h]
\centering
\labellist
\small\hair 2pt
\pinlabel \mbox{type I} at 65 0
\pinlabel \mbox{type II} at 235 0
\endlabellist
\includegraphics[scale=0.5]{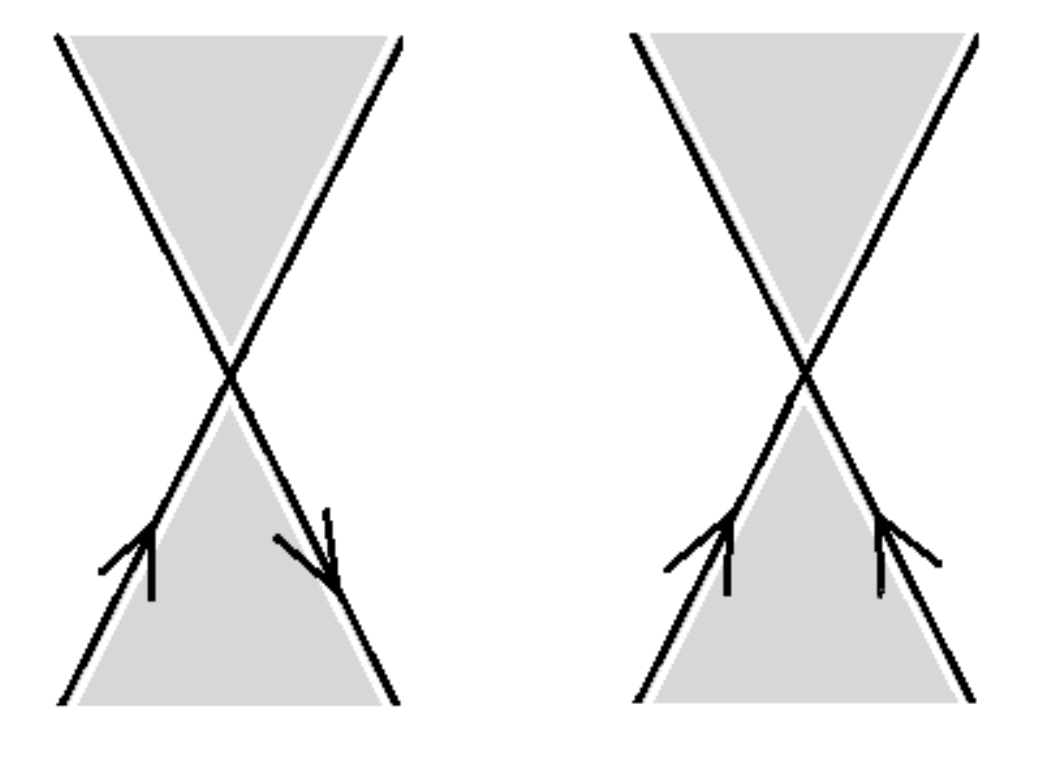}
\caption{Crossings in an oriented chessboard-shaded diagram}
\label{fig:over}
\end{figure} 
\begin{prop}\cite[Theorem 10.1]{mura}\cite[Corollary 3.9]{coch}\label{semn}
Let $L$ be an oriented link in $\mathbb{R}^3$. Then the unlinking number of $L$ satisfies
\begin{equation*}
u(L)\geq \frac{|\sigma(L)|}{2},
\end{equation*}
where $\sigma(L)$ is the signature of the link.
\end{prop}
\begin{proof}
Consider the trivial link with $k$ components and the standard diagram consisting of $k$ non-nested circles with no crossings. For one choice of shading, the corresponding Goeritz matrix $G$ of this link is the zero matrix with $k-1$ rows and columns, which has $\sign G=0$. Since there are no crossings in this diagram of the link, we have $\mu=0$. It follows from (\ref{ostea}) that the signature of the trivial link is $0$, irrespective of the number of components. Now, given an oriented link $L$ with diagram $D$, we aim to obtain the trivial link by changing crossings in $D$. At each step, let $c$ denote the crossing to be changed, and choose the chessboard colouring of the diagram that makes $c$ a double point of type I. Also, relabel the white regions so that $c$ is adjacent to $R_0$ and $R_n$. In the matrix $G'$ of the link, the effect of the crossing change amounts to changing entries $g_{00}$, $g_{0n}$, $g_{n0}$ and $g_{nn}$. Therefore, the new Goeritz matrix of the link is identical to the original one, except for the diagonal entry $g_{nn}$. By Lemma \ref{prima}, $\text{sign }G$ changes by at most $2$. Since $c$ is a double point of type I, changing the crossing will not affect $\mu$. It follows from (\ref{ostea}) that $\sigma(L)$, in turn, changes by at most $2$ . The link is eventually converted to the trivial link, so that its signature changes by at most twice the unlinking number throughout the process, which implies that $|\sigma(L)|\leq2u(L)$, or equivalently, $u(L)\geq|\sigma(L)|/2$.
\end{proof}   
To illustrate the application of this method, consider the link $L10a99$. Using (\ref{ostea}), one may show that the link has signature $-5$ when oriented as in Figure \ref{fig:semn}, so that $u(L10a99)\geq 3$. Therefore, the link has unlinking number $3$, as it can be converted to the trivial link with $2$ components by changing the $3$ crossings indicated in the figure.
\begin{figure}[h]
\centering
\includegraphics[scale=0.3]{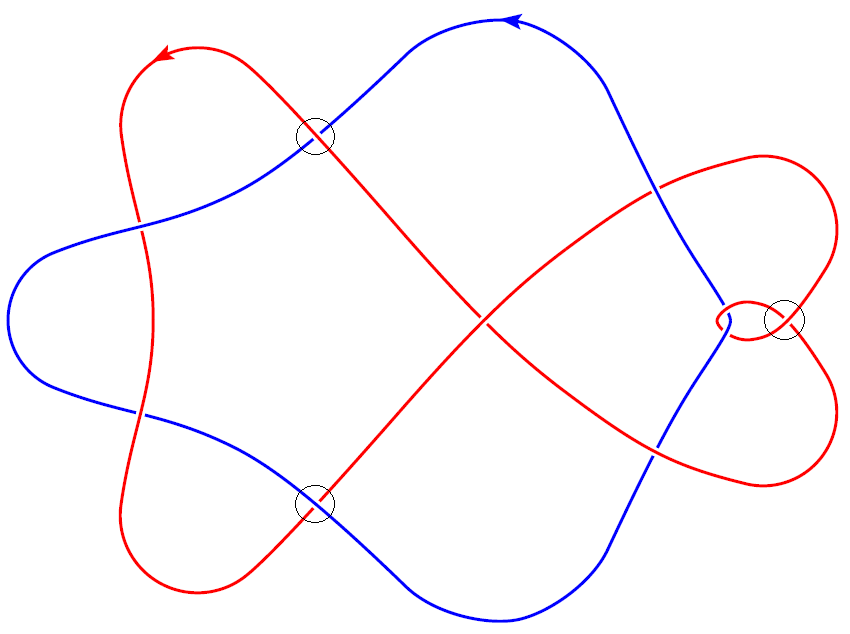}
\caption{One possible way to unlink $L10a99$}
\label{fig:semn}
\end{figure} 
\subsection{Link determinant and link nullity}
The \textit{determinant} of a link is defined to be the determinant of its Goeritz matrix. Similarly, the \textit{nullity} of a link is equal to the nullity of its Goeritz matrix, provided that a connected diagram is considered. 
\begin{prop}\cite[Corollary 3.21]{kauffman}\cite[Corollary 4.3]{kawa}\cite[Lemma 2.4]{owens}\label{curent}
Let $L$ be a link in $\mathbb{R}^3$, with $k$ components, nullity $\eta(L)$ and determinant $\det L$. Let $u(L)$ be the unlinking number of $L$. 
\begin{enumerate}[a)]
\item Then
\begin{equation*}
u(L)\geq k-1-\eta(L).
\end{equation*}
\item If $u(L)\leq k-1$, then $\det L=2^{k-1}c^2$, for some $c\in \mathbb{Z}$.
\end{enumerate}
\end{prop}
\begin{proof}
Consider the trivial link with $k$ components and a connected diagram consisting of $k$ circles sitting in a row, with two crossings between each adjacent pair of circles and no other crossings. For either choice of shading, the Goeritz matrix $G$ of this link is the zero matrix with $k-1$ rows and columns, which has nullity $k-1$. Now, given a diagram of a link $L$ with $k$ components and nullity $\eta(L)$, construct the matrix $G'$ as in Section \ref{signa} and change a crossing. As before, we can arrange so that the change affects only one entry in the Goeritz matrix of $L$, namely the bottom right element $g_{nn}$. It follows from Lemma \ref{prima} that the nullity of the Goeritz matrix will change by at most $1$, and so too will the nullity of the link. Since $L$ is converted to the trivial link with $u(L)$ crossing changes, its nullity cannot change by more than the unlinking number, giving
$u(L)\geq|(k-1)-\eta(L)|\geq k-1-\eta(L)$. For a proof of part b) see \cite{kawa}, where this statement is shown to follow from a stronger condition involving multivariable Alexander polynomials, or \cite{owens}.
\end{proof}
To illustrate the application of the method described in Proposition \ref{curent} part a), consider the link $L10a169$ with $4$ components and nullity $0$, so that $u(L10a169)\geq 3$. Therefore, the link has unlinking number $3$, as it can be converted to the trivial link with $4$ components by changing the $3$ crossings indicated in Figure \ref{fig:figu}.
\begin{figure}[h]
\centering
\includegraphics[scale=0.33]{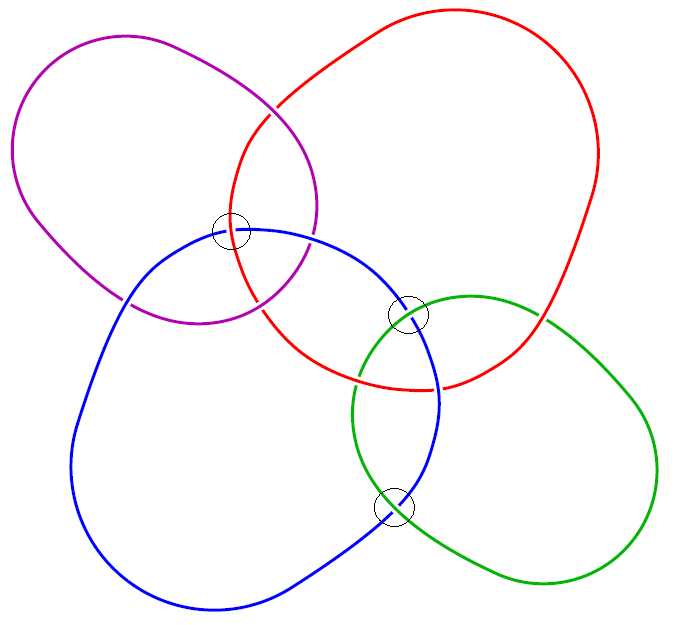}
\caption{One possible way to unlink $L10a169$}
\label{fig:figu}
\end{figure}

For the method described in part b), let $L$ be the link $L10n33$, with $k=2$ components and determinant $\det L=48$. Suppose that $u(L)\leq1$. Then by the proposition, $c^2=24$ for some $c\in\mathbb{Z}$, a contradiction that gives $u(L)>1$. Therefore, the link has unlinking number $2$, as it can be converted to the trivial link with $2$ components by changing the $2$ crossings indicated in Figure \ref{fig:triv}.
\begin{figure}[h]
\centering
\includegraphics[scale=0.32]{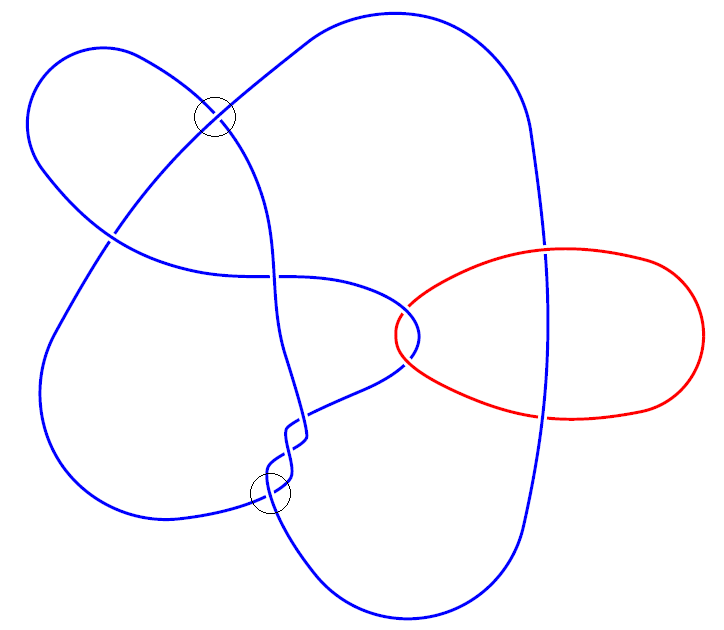}
\caption{One possible way to unlink $L10n33$}
\label{fig:triv}
\end{figure}

Every $n\times n$ integer matrix $M$ can be transformed by a finite sequence of row and column operations into a diagonal matrix, whose diagonal entries form a sequence $\{a_1,a_2,\dots,a_r,0,\dots,0\}$, where $a_i$ is nonnegative and $a_i$ divides $a_{i+1}$. This diagonal matrix is independent of the sequence of row and column operations, and is called the \textit{Smith normal form} of $M$. The matrix $M$ presents the quotient group $\mathbb{Z}^n/M\mathbb{Z}^n$, which is cyclic if and only if the Smith normal form $S$ of $M$ satisfies $s_{ii}=1$ for $i=1,\dots,n-1$, and $s_{nn}=\det M$.
\begin{prop}\cite[Lemma 4.1]{owens}\label{contr}
Let $L$ be a link with $2$ components in $\mathbb{R}^3$ and determinant $\det L$, such that its unlinking number satisfies $u(L)<3$. Suppose that the Goeritz matrix of $L$ presents a finite cyclic group. Then at least one of the following statements holds:
\begin{itemize}
\item $\det{L}$ is a multiple of 4, and the absolute value of at least one of the signatures of $L$ is 1,
\item $\det{L}$ is a multiple of 16,
\item $\det{L}=2t^2$, for some $t\in\mathbb{Z}$.
\end{itemize}
\end{prop}
The proof of this proposition is based on a $4$-dimensional manifold bounded by the double branched cover $Y$ of the link $L$. This gives constraints on the linking form of $Y$, which in turn gives constraints on the determinant and signature of $L$. For details see \cite{owens}.

To illustrate the application of this method, let $L$ be the link $L10a54$, with $2$ components and determinant $78$. The Smith normal form of the Goeritz matrix $G$ of $L$ is
\[
S=\begin{bmatrix}
    1      & 0 & 0 & 0 \\
   0     &1 &0 & 0 \\
 0&0&1&0\\
    0    & 0 & 0 &78
\end{bmatrix},
\]
so that $G$ presents a finite cyclic group. The determinant of $L$ is neither a multiple of $4$, nor a multiple of $16$, nor twice the square of some $t\in\mathbb{Z}$, so that $u(L)\geq3$ by Proposition \ref{contr}. Therefore, the link has unlinking number $3$, as it can be converted to the trivial link with $2$ components by changing the $3$ crossings indicated in Figure \ref{fig:propo}.
\begin{figure}[h]
\centering
\includegraphics[scale=0.32]{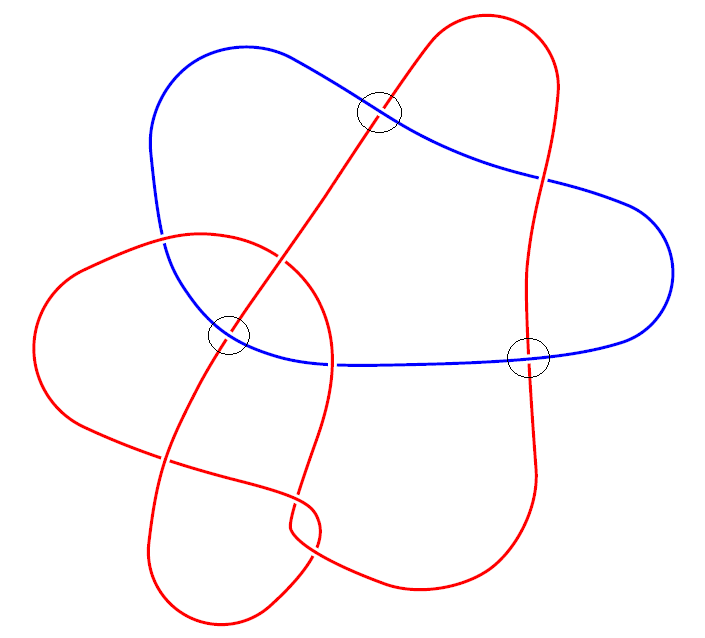}
\caption{One possible way to unlink $L10a54$}
\label{fig:propo}
\end{figure}

The following lemma can be viewed as a signed refinement of Proposition~\ref{curent}a.
\begin{lemma}\cite[Lemma 2.2]{owens}\label{lema2}
If an oriented link $L$ with $k$ components, signature $\sigma(L)$ and nullity $\eta(L)$ is converted to the trivial link by changing $p$ positive crossings and $n$ negative crossings in some diagram $D$ of the link, then
\begin{align*}
p&\geq\frac{-\sigma(L)-\eta(L)+k-1}{2}.
\end{align*}
\end{lemma}
\begin{proof}
Let $c$ be a positive crossing in the diagram of $L$, and choose the chessboard colouring of $D$ that makes $c$ a double point of type I. In this situation, $c$ has incidence number $\iota(c)=-1$. Let $G$ be the Goeritz matrix of the diagram and suppose we change the crossing $c$. As in the proof of Proposition~\ref{semn}, we are free to relabel the white regions, so that the new Goeritz matrix of the link is identical to the original one, except for one diagonal entry. After the change, $c$ is still a double point of type I, but its incidence number becomes $\iota(c)=1$. Therefore, the diagonal entry that distinguishes between the two Goeritz matrices increases. By Lemma~\ref{prima}, if the nullity of $G$ stays the same, then the signature of $G$ either stays the same or increases by $2$, and following (\ref{ostea}), so too does $\sigma(L)+\eta(L)$. If the nullity changes, it can only be by $1$, in which case Lemma~\ref{prima} tells us that the signature of $G$ increases by $1$, and consequently, $\sigma(L)+\eta(L)$ stays the same or increases by $2$. By a similar argument, changing a negative crossing causes $\sigma(L)+\eta(L)$ to either stay constant or decrease by $2$. As we have seen previously, the signature and nullity of the trivial link with $k$ components add up to $k-1$. The link $L$ is eventually converted to the trivial link, so that $\sigma(L)+\eta(L)$ increases by at most twice the number of positive crossings we change, giving $(k-1)-(\sigma(L)+\eta(L))\leq2p$, or equivalently,
\begin{equation*}
p\geq\frac{-\sigma(L)-\eta(L)+k-1}{2},
\end{equation*}
as required.
\end{proof} 
\subsection{Lattice embeddings}
Let the set of vectors $\{\textbf{a}_1,\dots,\textbf{a}_n\}$ form a basis for $\mathbb{R}^n$ over $\mathbb{R}$. These vectors span a \textit{lattice} $\Lambda$, which is the set of all linear combinations $\{m_1\textbf{a}_1+\dots+m_n\textbf{a}_n\}$ with $m_i\in\mathbb{Z}$, $i=1,\dots,n$. Let $\{\textbf{b}_1,\dots,\textbf{b}_k\}$ be a set of vectors in $\Lambda$. These vectors span a \textit{sublattice} $\Lambda_b\subset\Lambda$, which is the set of all linear combinations $\{n_1\textbf{b}_1+\dots+n_k\textbf{b}_k\}$ with $n_j\in\mathbb{Z}$, $j=1,\dots,k$. The sublattice $\Lambda_b$ of $\Lambda$ is called \textit{primitive} if for all $\textbf{v}\in\Lambda$ and for all $m\in\mathbb{N}$, if $m\textbf{v}\in\Lambda_b$ then $\textbf{v}\in\Lambda_b$. Nagel and Owens gave an obstruction to equality in the lower bound from Lemma \ref{lema2}, which we describe next. 
\begin{prop}\cite[Corollary 3]{owens}\label{gap}
Let $L$ be an oriented non-split alternating link, with $k$ components and signature $\sigma(L)$. Suppose $L$ can be converted to the trivial link by changing $p=\frac{-\sigma(L)+k-1}{2}$ positive crossings and $n$ negative crossings in some diagram of $L$. Let $m$ be the rank of the positive-definite Goeritz matrix $G$ associated to an alternating diagram of $L$, and define $l=m+2(n+p)-k+1$. Then $G$ admits a factorisation as $A^TA$, where $A$ is an integer $l\times m$ matrix. Moreover, there exist vectors \textbf{v}$_i$ for $i=1,\dots,p+n$ in $(\Col A)^{\perp}\subset \mathbb{Z}^l$ spanning a primitive sublattice of $\mathbb{Z}^l$, such that $\textbf{v}_i\cdot\textbf{v}_j=2\delta_{ij}$, where $\delta_{ij}$ is the Kronecker delta.
\end{prop}
The proof of Proposition \ref{gap} uses results of Gordon and Litherland in \cite{litherland}, as well as the celebrated Diagonalisation Theorem of Donaldson in \cite{donaldson}, and is based on a generalisation of earlier work by Cochran and Lickorish in \cite{coch}. 

To illustrate the application of this method, let $L$ be the link $L10a138$, with $3$ components, determinant $48$ and nullity $0$. By part b) of Proposition~\ref{curent}, $u(L)>2$, and we aim to obstruct it from being $3$. When oriented as in Figure~\ref{fig:gap}, the link has signature $-4$. Suppose $L$ can be converted to the trivial link by changing $p$ positive crossings and $n$ negative crossings in some diagram. By Lemma~\ref{lema2}, $p\geq3$. Thus the only possibility if $u(L)=3$ is to have $p=3$ and $n=0$, which we will show cannot occur. Suppose $p=3$ and $n=0$. The positive-definite Goeritz matrix of the chosen alternating diagram is
\[
G=\begin{bmatrix}
    7      & -1 & -1 \\
   -1     &3 &-1 \\
 -1&-1&3\\
\end{bmatrix},
\]
which has rank $m=3$. Keeping the notation in Proposition~\ref{gap}, we have $l=7$. For any factorisation of $G$ as $A^TA$, where $A$ is a $7\times3$ integer matrix and $A^T$ is its transpose, another may be obtained by interchanging the second and third columns of $A$, permuting the rows of $A$, or multiplying a subset of the rows of $A$ by $-1$. Up to these symmetries, we are left with $9$ solutions, as follows:
\[
\begin{bmatrix}
-1&1&1\\
2&0&0\\
1&-1&1\\
-1&-1&1\\
0&0&0\\
0&0&0\\
0&0&0
\end{bmatrix},
\begin{bmatrix}
-1&1&1\\
2&0&0\\
0&-1&1\\
0&-1&1\\
1&0&0\\
1&0&0\\
0&0&0
\end{bmatrix},
\begin{bmatrix}
1&1&0\\
1&-1&1\\
-1&1&0\\
-1&0&1\\
-1&0&1\\
1&0&0\\
1&0&0
\end{bmatrix},
\begin{bmatrix}
-1&1&1\\
1&-1&1\\
-1&-1&1\\
1&0&0\\
1&0&0\\
1&0&0\\
1&0&0
\end{bmatrix},
\begin{bmatrix}
-2&1&0\\
1&0&1\\
-1&-1&1\\
-1&0&1\\
0&1&0\\
0&0&0\\
0&0&0
\end{bmatrix},
\]
\[
\begin{bmatrix}
-2&1&0\\
1&1&0\\
0&-1&1\\
-1&0&1\\
0&0&1\\
1&0&0\\
0&0&0
\end{bmatrix},
\begin{bmatrix}
2&0&0\\
1&-1&1\\
-1&0&1\\
-1&0&1\\
0&1&0\\
0&1&0\\
0&0&0
\end{bmatrix},
\begin{bmatrix}
-2&1&0\\
-1&-1&1\\
0&1&0\\
0&0&1\\
0&0&1\\
1&0&0\\
1&0&0
\end{bmatrix},
\begin{bmatrix}
2&0&0\\
0&-1&1\\
-1&1&0\\
-1&0&1\\
0&1&0\\
0&0&1\\
1&0&0
\end{bmatrix}.
\]
It is straightforward to check that for any matrix $A$ in this list, there does not exist a set of vectors $\{\textbf{v}_1,\textbf{v}_2,\textbf{v}_3\}$ in the orthogonal complement of the column space of $A$, such that $\textbf{v}_i\cdot\textbf{v}_j=2\delta_{ij}$, so that $u(L)\geq4$ by Proposition \ref{gap}. Therefore, the link has unlinking number $4$, as it can be converted to the trivial link with $3$ components by changing the $4$ crossings indicated in Figure~\ref{fig:gap}.

In general, the method based on Proposition~\ref{gap} gives a somewhat involved algorithm to obstruct equality in Lemma \ref{lema2}, leading to improved lower bounds on unlinking number. All possible factorisations of the Goeritz matrix can be found by hand, but this can also be done using the command $OrthogonalEmbeddings$ provided by GAP \cite{gap}. 
\begin{figure}[h]
\centering
\includegraphics[scale=0.23]{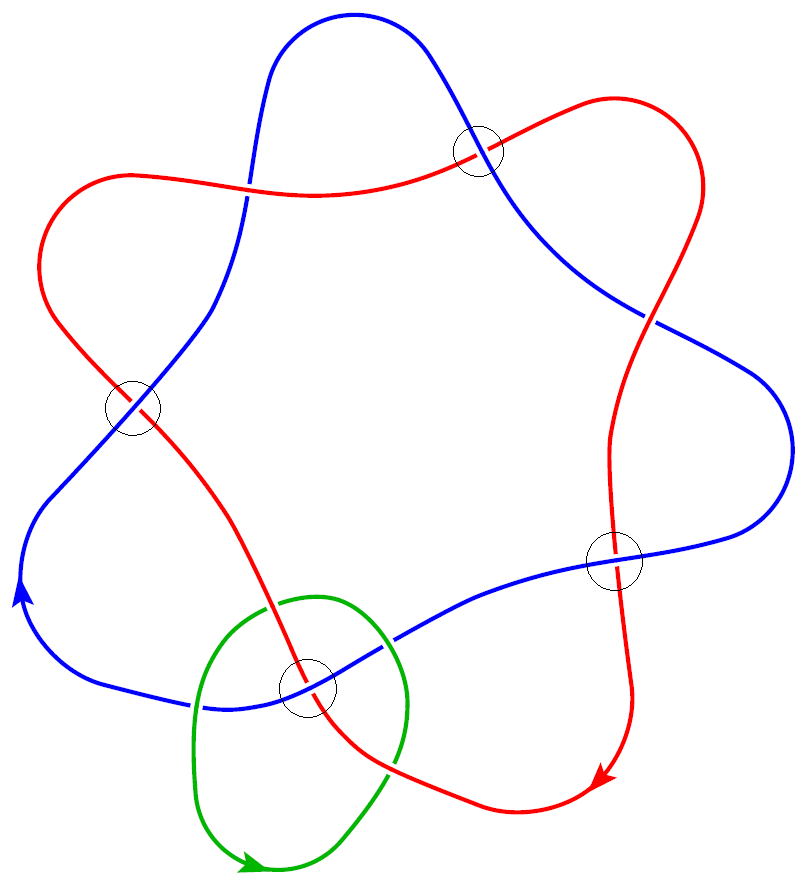}
\caption{One possible way to unlink $L10a138$}
\label{fig:gap}
\end{figure}

So far, five methods that give lower bounds on the unlinking number of a link --- alone or combined --- have been described in Propositions \ref{primpropo}, \ref{semn}, \ref{curent}, \ref{contr} and \ref{gap}. The next method was developed by Kohn in \cite{kohn}. 
\subsection{Covering links}
Let $p:\mathbb{C}\times\mathbb{R}\rightarrow\mathbb{C}\times\mathbb{R}$ be the map taking $(z,t)$ to $(z^2,t)$. Let $L$ be a link with $2$ components, say $L=A\sqcup B$, where $A$ is the trivial knot and $lk(A,B)=0$. Assume, after isotopy in $S^3=\mathbb{R}^3\cup{\{\infty\}}=\mathbb{C}\times\mathbb{R}\cup{\{\infty\}}$, that $A$ is $0\times\mathbb{R}$, and let $\tilde{B}$ be the preimage of $B$ under $p$. We refer to $\tilde{B}$ as the \textit{covering link} of $B$ under $p$.
\begin{prop}\cite[Method 5]{kohn}\label{prop6}
Let $L$ be a link with $2$ components, say $A$ and $B$, such that $A$ is the trivial knot and $lk(A,B)=0$. If $L$ is unlinked by a single crossing change involving $B$ only, then the unlinking number of $\tilde{B}$ is at most $2$.
\end{prop}
\begin{hproof}
Suppose $L$ can be converted to the trivial link by changing a single crossing $c$, with both strands of $c$ belonging to component $B$. We may isotope $B$ so that it lies near the plane $\mathbb{C}\times\{0\}$ and its projection onto this plane contains the unlinking crossing $c$. The preimage $\tilde{B}$ of $B$ will then contain two crossings $c_1$ and $c_2$, which are the preimage of $c$ under $p$. Changing $c$ converts $L$ to the unlink, therefore changing $c_1$ and $c_2$ must convert $\tilde{B}$ to the unlink, since the preimage under $p$ of a circle in $\mathbb{C}\times\{0\}$ not containing the origin is a pair of circles.
\end{hproof}
To illustrate the application of this method, let $L$ be the link $L10a7$ shown in Figure~\ref{fig:incauna}. The link has $2$ components, namely the red trivial knot $A$ and the blue figure-eight knot $B$, with $lk(A,B)=0$. If $L$ can be converted to the unlink by changing a single crossing, then both strands must belong to the knotted component $B$. So suppose that $L$ is converted to the unlink by a single crossing change involving $B$ only. After isotopy, assume that $A$ is $0\times\mathbb{R}$, as depicted in Figure~\ref{fig:poza}. 
\begin{figure}[t]
\centering
\includegraphics[scale=0.63]{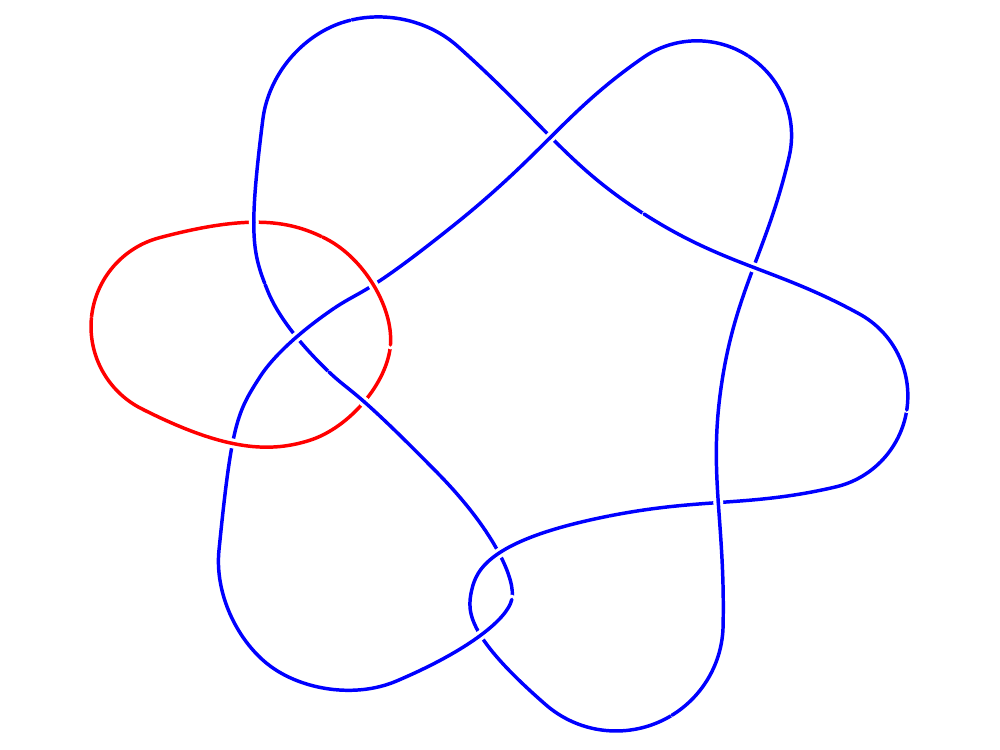}
\caption{Diagram of $L10a7$}
\label{fig:incauna}
\end{figure}
\begin{figure}[H]
\centering
\includegraphics[scale=0.3]{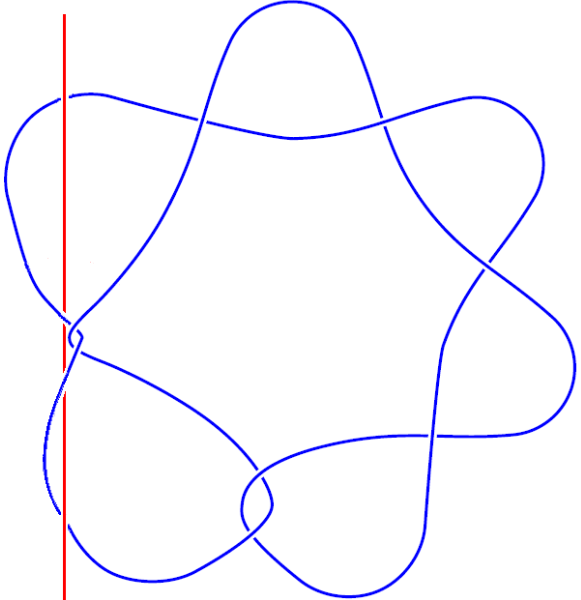}
\caption{$L10a7$ when the trivial component is $0\times\mathbb{R}$}
\label{fig:poza}
\end{figure}

The preimage of $B$ under the map $p$ is the union of $B$ and its rotated image, glued together to form the covering link $\tilde{B}$, as in Figure~\ref{finala}. It consists of two Stevedore knots --- each with unknotting number $1$ --- with linking number $2$ when oriented as shown.
\begin{figure}[h]
\centering
\includegraphics[scale=0.33]{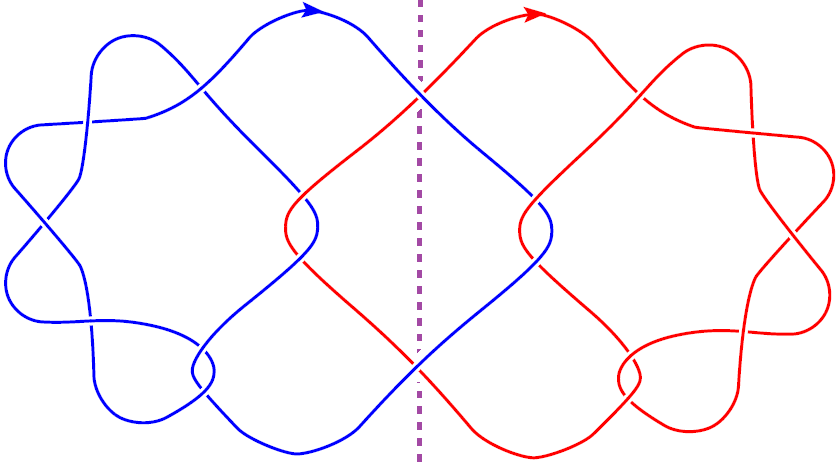}
\caption{Diagram of $\tilde{B}$}
\label{finala}
\end{figure}

Following Proposition~\ref{primpropo}, $u(\tilde{B})\geq4$, contradicting Proposition \ref{prop6}. Therefore, $u(L)\geq2$, and $L$ has unlinking number $2$, as it can be converted to the trivial link with $2$ components by changing the $2$ crossings indicated in Figure~\ref{fig:ciudat}.
\begin{figure}[h]
\centering
\includegraphics[scale=0.29]{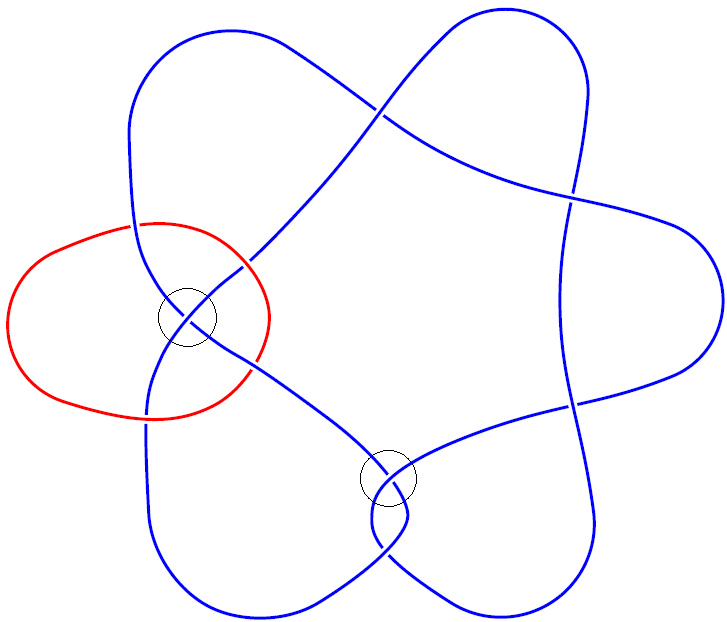}
\caption{One possible way to unlink $L10a7$}
\label{fig:ciudat}
\end{figure}
\section{Table of unlinking numbers}\label{table}
Table \ref{name} contains all prime, non-split links with crossing number $10$ and at least $2$ components, together with the unlinking number $u(L)$ of each link and a proposition that gives a lower bound that realises $u(L)$. With the exception of $L10n32$ and $L10n34$, the table is complete. 
\subsection{Unknown cases}
Although the methods in this paper were not sufficient to determine the unlinking numbers of $2$ of the links in the table, they still provide partial information. In the following, $p$ is the number of positive crossings and $n$ is the number of negative crossings that we change. 
\begin{itemize}
\item $L10n32$ has $u(L)\geq1$ and we conjecture that $u(L)=2$; 
\item $L10n34$ has $u(L)\geq2$ by Proposition~\ref{curent} and we conjecture that $u(L)=3$; the cases $p=0$, $n=2$ and $p=1$, $n=1$ for any choice of orientation are obstructed by Lemma \ref{lema2} and Proposition~\ref{gap}, respectively.
\end{itemize}
\clearpage
\begin{figure}[h]
\centering
\labellist
\small\hair 2pt
\pinlabel $L10n32$ at 160 10
\pinlabel $L10n34$ at 620 10
\endlabellist
\includegraphics[scale=0.4]{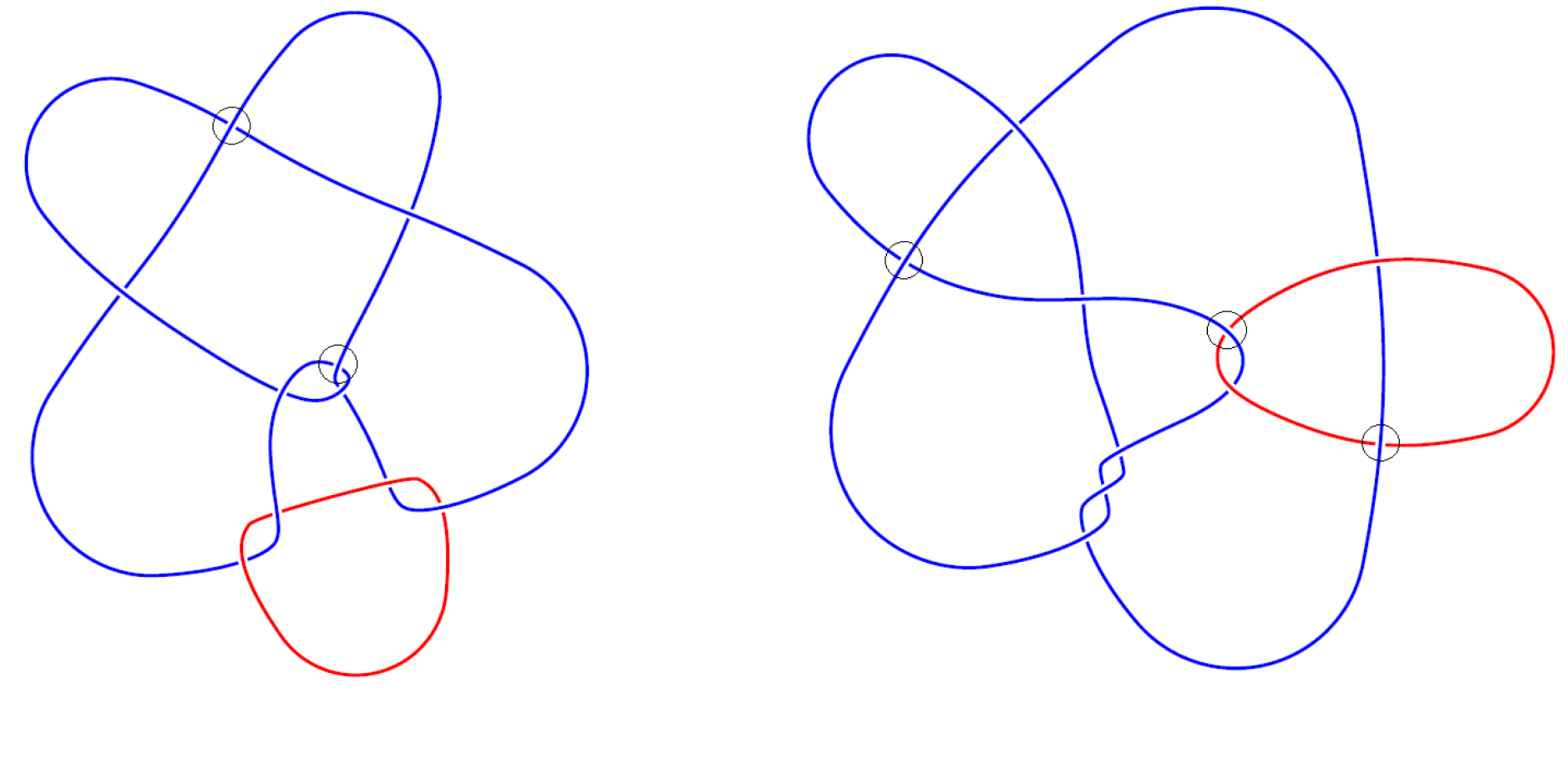}
\caption{Showing a set of crossing changes that unlink the two remaining links}
\label{fig:tot}
\end{figure}
\nocite{*}
\begin{table}[h]
\begin{center}
\small
\begin{tabular}[htb]{l c l}
\toprule
Link $L$ & $u(L)$ &Method \\
\midrule
$L10a1$ & $2$ & Prop \ref{curent}b\\
$L10a2$ & $2$ & Prop \ref{semn}\\
$L10a3$ & $2$ & Prop \ref{prop6} \& \ref{primpropo}\\
$L10a4$ & $2$ & Prop \ref{curent}b\\
$L10a5$ & $2$ & Prop \ref{curent}b\\
$L10a6$ & $2$ & Prop \ref{curent}b\\
$L10a7$ & $2$ & Prop \ref{prop6} \& \ref{primpropo}\\
$L10a8$ & $3$ & Prop \ref{gap}\\
$L10a9$ & $2$ & Prop \ref{semn}\\
$L10a10$ & $2$ & Prop \ref{curent}b\\
$L10a11$ & $3$ & Prop \ref{primpropo}\\
$L10a12$ & $3$ & Prop \ref{primpropo}\\
$L10a13$ & $3$ & Prop \ref{primpropo}\\
$L10a14$ & $2$ & Prop \ref{curent}b\\
$L10a15$ & $3$ & Prop \ref{primpropo}\\
$L10a16$ & $3$ & Prop \ref{primpropo}\\
$L10a17$ & $3$ & Prop \ref{gap}\\
$L10a18$ & $2$ & Prop \ref{primpropo}\\
$L10a19$ & $2$ & Prop \ref{curent}b\\
$L10a20$ & $2$ & Prop \ref{curent}b\\
$L10a21$ & $2$ & Prop \ref{curent}b\\
$L10a22$ & $2$ & Prop \ref{semn}\\
$L10a23$ & $3$ & Prop \ref{gap}\\
$L10a24$ & $3$ & Prop \ref{gap}\\
$L10a25$ & $3$ & Prop \ref{primpropo}\\
$L10a26$ & $3$ & Prop \ref{primpropo}\\
$L10a27$ & $2$ & Prop \ref{semn}\\
$L10a28$ & $1$ & Prop \ref{primpropo}\\
$L10a29$ & $1$ & Prop \ref{primpropo}\\
$L10a30$ & $3$ & Prop \ref{primpropo}\\
$L10a31$ & $2$ & Prop \ref{curent}b\\
$L10a32$ & $2$ & Prop \ref{prop6} \& \ref{gap}\\
$L10a33$ & $3$ & Prop \ref{primpropo}\\
$L10a34$ & $1$ & Prop \ref{primpropo}\\
$L10a35$ & $2$ & Prop \ref{primpropo}\\
$L10a36$ & $1$ & Prop \ref{primpropo}\\
$L10a37$ & $3$ & Prop \ref{primpropo}\\
$L10a38$ & $4$ & Prop \ref{primpropo}\\
$L10a39$ & $2$ & Prop \ref{primpropo}\\
$L10a40$ & $3$ & Prop \ref{primpropo}\\
$L10a41$ & $2$ & Prop \ref{curent}b\\
$L10a42$ & $2$ & Prop \ref{primpropo}\\
$L10a43$ & $4$ & Prop \ref{primpropo}\\
$L10a44$ & $4$ & Prop \ref{primpropo}\\
$L10a45$ & $3$ & Prop \ref{primpropo}\\
$L10a46$ & $4$ & Prop \ref{primpropo}\\
$L10a47$ & $3$ & Prop \ref{primpropo}\\
$L10a48$ & $2$ & Prop \ref{primpropo}\\
\bottomrule
\end{tabular}
\qquad
\begin{tabular}[tbp]{l c l}
\toprule
Link $L$ & $u(L)$ &Method \\
\midrule
$L10a49$ & $4$ & Prop \ref{primpropo}\\
$L10a50$ & $3$ & Prop \ref{primpropo}\\
$L10a51$ & $1$ & Prop \ref{primpropo}\\
$L10a52$ & $2$ & Prop \ref{semn}\\
$L10a53$ & $1$ & Prop \ref{primpropo}\\
$L10a54$ & $3$ & Prop \ref{contr}\\
$L10a55$ & $2$ & Prop \ref{curent}b\\
$L10a56$ & $2$ & Prop \ref{curent}b\\
$L10a57$ & $2$ & Prop \ref{primpropo}\\
$L10a58$ & $4$ & Prop \ref{primpropo}\\
$L10a59$ & $2$ & Prop \ref{primpropo}\\
$L10a60$ & $2$ & Prop \ref{primpropo}\\
$L10a61$ & $2$ & Prop \ref{primpropo}\\
$L10a62$ & $3$ & Prop \ref{semn}\\
$L10a63$ & $3$ & Prop \ref{contr}\\
$L10a64$ & $2$ & Prop \ref{semn}\\
$L10a65$ & $2$ & Prop \ref{curent}b\\
$L10a66$ & $2$ & Prop \ref{primpropo}\\
$L10a67$ & $4$ & Prop \ref{primpropo}\\
$L10a68$ & $2$ & Prop \ref{primpropo}\\
$L10a69$ & $2$ & Prop \ref{primpropo}\\
$L10a70$ & $2$ & Prop \ref{semn}\\
$L10a71$ & $2$ & Prop \ref{curent}b\\
$L10a72$ & $4$ & Prop \ref{primpropo}\\
$L10a73$ & $3$ & Prop \ref{primpropo}\\
$L10a74$ & $4$ & Prop \ref{primpropo}\\
$L10a75$ & $3$ & Prop \ref{primpropo}\\
$L10a76$ & $2$ & Prop \ref{primpropo}\\
$L10a77$ & $4$ & Prop \ref{primpropo}\\
$L10a78$ & $4$ & Prop \ref{primpropo}\\
$L10a79$ & $2$ & Prop \ref{primpropo}\\
$L10a80$ & $2$ & Prop \ref{primpropo}\\
$L10a81$ & $4$ & Prop \ref{primpropo}\\
$L10a82$ & $3$ & Prop \ref{contr}\\
$L10a83$ & $3$ & Prop \ref{primpropo}\\
$L10a84$ & $2$ & Prop \ref{primpropo}\\
$L10a85$ & $4$ & Prop \ref{primpropo}\\
$L10a86$ & $2$ & Prop \ref{primpropo}\\
$L10a87$ & $3$ & Prop \ref{primpropo}\\
$L10a88$ & $2$ & Prop \ref{primpropo}\\
$L10a89$ & $2$ & Prop \ref{semn}\\
$L10a90$ & $2$ & Prop \ref{curent}b\\
$L10a91$ & $2$ & Prop \ref{curent}b\\
$L10a92$ & $2$ & Prop \ref{primpropo}\\
$L10a93$ & $3$ & Prop \ref{gap}\\
$L10a94$ & $4$ & Prop \ref{primpropo}\\
$L10a95$ & $1$ & Prop \ref{primpropo}\\
$L10a96$ & $4$ & Prop \ref{primpropo}\\
\bottomrule
\end{tabular}
\qquad
\begin{tabular}[tbp]{l c l}
\toprule
Link $L$ & $u(L)$ &Method \\
\midrule
$L10a97$ & $4$ & Prop \ref{primpropo}\\
$L10a98$ & $4$ & Prop \ref{primpropo}\\
$L10a99$ & $3$ & Prop \ref{semn}\\
$L10a100$ & $4$ & Prop \ref{primpropo}\\
$L10a101$ & $4$ & Prop \ref{primpropo}\\
$L10a102$ & $4$ & Prop \ref{primpropo}\\
$L10a103$ & $1$ & Prop \ref{semn}\\
$L10a104$ & $2$ & Prop \ref{primpropo}\\
$L10a105$ & $4$ & Prop \ref{primpropo}\\
$L10a106$ & $3$ & Prop \ref{gap}\\
$L10a107$ & $4$ & Prop \ref{primpropo}\\
$L10a108$ & $4$ & Prop \ref{primpropo}\\
$L10a109$ & $2$ & Prop \ref{primpropo}\\
$L10a110$ & $4$ & Prop \ref{primpropo}\\
$L10a111$ & $2$ & Prop \ref{curent}b\\
$L10a112$ & $2$ & Prop \ref{curent}b\\
$L10a113$ & $3$ & Prop \ref{gap}\\
$L10a114$ & $5$ & Prop \ref{primpropo}\\
$L10a115$ & $5$ & Prop \ref{primpropo}\\
$L10a116$ & $5$ & Prop \ref{primpropo}\\
$L10a117$ & $5$ & Prop \ref{primpropo}\\
$L10a118$ & $5$ & Prop \ref{primpropo}\\
$L10a119$ & $5$ & Prop \ref{primpropo}\\
$L10a120$ & $5$ & Prop \ref{primpropo}\\
$L10a121$ & $5$ & Prop \ref{primpropo}\\
$L10a122$ & $4$ & Prop \ref{primpropo}\\
$L10a123$ & $4$ & Prop \ref{primpropo}\\
$L10a124$ & $4$ & Prop \ref{primpropo}\\
$L10a125$ & $4$ & Prop \ref{primpropo}\\
$L10a126$ & $3$ & Prop \ref{primpropo}\\
$L10a127$ & $3$ & Prop \ref{curent}b\\
$L10a128$ & $3$ & Prop \ref{primpropo}\\
$L10a129$ & $3$ & Prop \ref{primpropo}\\
$L10a130$ & $4$ & Prop \ref{primpropo}\\
$L10a131$ & $4$ & Prop \ref{primpropo}\\
$L10a132$ & $4$ & Prop \ref{primpropo}\\
$L10a133$ & $4$ & Prop \ref{primpropo}\\
$L10a134$ & $4$ & Prop \ref{primpropo}\\
$L10a135$ & $3$ & Prop \ref{primpropo}\\
$L10a136$ & $2$ & Prop \ref{primpropo}\\
$L10a137$ & $4$ & Prop \ref{gap}\\
$L10a138$ & $4$ & Prop \ref{gap}\\
$L10a139$ & $4$ & Prop \ref{primpropo}\\
$L10a140$ & $2$ & Prop \ref{curent}a\\
$L10a141$ & $3$ & Prop \ref{gap}\\
$L10a142$ & $5$ & Prop \ref{primpropo}\\
$L10a143$ & $5$ & Prop \ref{primpropo}\\
$L10a144$ & $5$ & Prop \ref{primpropo}\\
\bottomrule
\end{tabular}
\caption{Unlinking numbers of prime, non-split links with crossing number $10$}
\label{name}
\end{center}
\end{table}

\begin{table}
\renewcommand\thetable{1}
\begin{center}
\small
\begin{tabular}[tbp]{l c l}
\toprule
Link $L$ & $u(L)$ &Method \\
\midrule
$L10a145$ & $5$ & Prop \ref{primpropo}\\
$L10a146$ & $5$ & Prop \ref{primpropo}\\
$L10a147$ & $3$ & Prop \ref{primpropo}\\
$L10a148$ & $3$ & Prop \ref{primpropo}\\
$L10a149$ & $3$ & Prop \ref{primpropo}\\
$L10a150$ & $3$ & Prop \ref{primpropo}\\
$L10a151$ & $3$ & Prop \ref{curent}b\\
$L10a152$ & $5$ & Prop \ref{primpropo}\\
$L10a153$ & $5$ & Prop \ref{primpropo}\\
$L10a154$ & $4$ & Prop \ref{primpropo}\\
$L10a155$ & $4$ & Prop \ref{primpropo}\\
$L10a156$ & $2$ & Prop \ref{primpropo}\\
$L10a157$ & $4$ & Prop \ref{gap}\\
$L10a158$ & $4$ & Prop \ref{gap}\\
$L10a159$ & $5$ & Prop \ref{primpropo}\\
$L10a160$ & $5$ & Prop \ref{primpropo}\\
$L10a161$ & $5$ & Prop \ref{primpropo}\\
$L10a162$ & $3$ & Prop \ref{primpropo}\\
$L10a163$ & $3$ & Prop \ref{curent}b\\
$L10a164$ & $5$ & Prop \ref{primpropo}\\
$L10a165$ & $4$ & Prop \ref{primpropo}\\
$L10a166$ & $5$ & Prop \ref{primpropo}\\
$L10a167$ & $5$ & Prop \ref{primpropo}\\
$L10a168$ & $5$ & Prop \ref{primpropo}\\
$L10a169$ & $3$ & Prop \ref{semn}\\
$L10a170$ & $4$ & Prop \ref{primpropo}\\
$L10a171$ & $5$ & Prop \ref{primpropo}\\
$L10a172$ & $5$ & Prop \ref{primpropo}\\
$L10a173$ & $5$ & Prop \ref{primpropo}\\
$L10a174$ & $5$ & Prop \ref{primpropo}\\
$L10n1$ & $3$ & Prop \ref{primpropo}\\
$L10n2$ & $1$ & Prop \ref{primpropo}\\
$L10n3$ & $2$ & Prop \ref{curent}b\\
$L10n4$ & $3$ & Prop \ref{primpropo}\\
$L10n5$ & $2$ & Prop \ref{semn}\\
$L10n6$ & $2$ & Prop \ref{curent}b\\
$L10n7$ & $3$ & Prop \ref{primpropo}\\
$L10n8$ & $2$ & Prop \ref{curent}b\\
$L10n9$ & $1$ & Prop \ref{primpropo}\\
$L10n10$ & $3$ & Prop \ref{primpropo}\\
$L10n11$ & $1$ & Prop \ref{primpropo}\\
$L10n12$ & $2$ & Prop \ref{curent}b\\
$L10n13$ & $3$ & Prop \ref{primpropo}\\
$L10n14$ & $1$ & Prop \ref{primpropo}\\
$L10n15$ & $3$ & Prop \ref{contr}\\
$L10n16$ & $2$ & Prop \ref{primpropo}\\
$L10n17$ & $3$ & Prop \ref{primpropo}\\
$L10n18$ & $1$ & Prop \ref{primpropo}\\
\bottomrule
\end{tabular}
\qquad
\begin{tabular}[tbp]{l c l}
\toprule
Link $L$ & $u(L)$ &Method \\
\midrule
$L10n19$ & $3$ & Prop \ref{primpropo}\\
$L10n20$ & $2$ & Prop \ref{curent}b\\
$L10n21$ & $1$ & Prop \ref{semn}\\
$L10n22$ & $1$ & Prop \ref{primpropo}\\
$L10n23$ & $3$ & Prop \ref{semn}\\
$L10n24$ & $2$ & Prop \ref{curent}b\\
$L10n25$ & $3$ & Prop \ref{primpropo}\\
$L10n26$ & $2$ & Prop \ref{primpropo}\\
$L10n27$ & $2$ & Prop \ref{primpropo}\\
$L10n28$ & $3$ & Prop \ref{primpropo}\\
$L10n29$ & $3$ & Prop \ref{primpropo}\\
$L10n30$ & $3$ & Prop \ref{primpropo}\\
$L10n31$ & $3$ & Prop \ref{primpropo}\\
$L10n32$ & $[1,2]$ & $-$\\
$L10n33$ & $2$ & Prop \ref{curent}b\\
$L10n34$ & $[2,3]$ & Prop \ref{curent}b\\
$L10n35$ & $2$ & Prop \ref{primpropo}\\
$L10n36$ & $2$ & Prop \ref{primpropo}\\
$L10n37$ & $4$ & Prop \ref{primpropo}\\
$L10n38$ & $4$ & Prop \ref{primpropo}\\
$L10n39$ & $3$ & Prop \ref{semn}\\
$L10n40$ & $2$ & Prop \ref{primpropo}\\
$L10n41$ & $2$ & Prop \ref{curent}b\\
$L10n42$ & $3$ & Prop \ref{semn}\\
$L10n43$ & $2$ & Prop \ref{primpropo}\\
$L10n44$ & $1$ & Prop \ref{primpropo}\\
$L10n45$ & $2$ & Prop \ref{primpropo}\\
$L10n46$ & $4$ & Prop \ref{primpropo}\\
$L10n47$ & $4$ & Prop \ref{primpropo}\\
$L10n48$ & $2$ & Prop \ref{primpropo}\\
$L10n49$ & $4$ & Prop \ref{primpropo}\\
$L10n50$ & $3$ & Prop \ref{contr}\\
$L10n51$ & $4$ & Prop \ref{primpropo}\\
$L10n52$ & $2$ & Prop \ref{primpropo}\\
$L10n53$ & $3$ & Prop \ref{primpropo}\\
$L10n54$ & $3$ & Prop \ref{semn}\\
$L10n55$ & $4$ & Prop \ref{primpropo}\\
$L10n56$ & $1$ & Prop \ref{semn}\\
$L10n57$ & $1$ & Prop \ref{curent}a\\
$L10n58$ & $2$ & Prop \ref{primpropo}\\
$L10n59$ & $2$ & Prop \ref{primpropo}\\
$L10n60$ & $4$ & Prop \ref{primpropo}\\
$L10n61$ & $4$ & Prop \ref{primpropo}\\
$L10n62$ & $3$ & Prop \ref{semn}\\
$L10n63$ & $3$ & Prop \ref{gap}\\
$L10n64$ & $2$ & Prop \ref{semn}\\
$L10n65$ & $4$ & Prop \ref{primpropo}\\
$L10n66$ & $4$ & Prop \ref{primpropo}\\
\bottomrule
\end{tabular}
\qquad
\begin{tabular}[tbp]{l c l}
\toprule
Link $L$ & $u(L)$ &Method \\
\midrule
$L10n67$ & $4$ & Prop \ref{primpropo}\\
$L10n68$ & $4$ & Prop \ref{primpropo}\\
$L10n69$ & $4$ & Prop \ref{primpropo}\\
$L10n70$ & $2$ & Prop \ref{primpropo}\\
$L10n71$ & $4$ & Prop \ref{primpropo}\\
$L10n72$ & $4$ & Prop \ref{primpropo}\\
$L10n73$ & $2$ & Prop \ref{primpropo}\\
$L10n74$ & $5$ & Prop \ref{primpropo}\\
$L10n75$ & $5$ & Prop \ref{primpropo}\\
$L10n76$ & $3$ & Prop \ref{primpropo}\\
$L10n77$ & $5$ & Prop \ref{primpropo}\\
$L10n78$ & $5$ & Prop \ref{primpropo}\\
$L10n79$ & $3$ & Prop \ref{primpropo}\\
$L10n80$ & $4$ & Prop \ref{primpropo}\\
$L10n81$ & $5$ & Prop \ref{primpropo}\\
$L10n82$ & $5$ & Prop \ref{primpropo}\\
$L10n83$ & $3$ & Prop \ref{primpropo}\\
$L10n84$ & $5$ & Prop \ref{primpropo}\\
$L10n85$ & $3$ & Prop \ref{primpropo}\\
$L10n86$ & $3$ & Prop \ref{primpropo}\\
$L10n87$ & $5$ & Prop \ref{primpropo}\\
$L10n88$ & $4$ & Prop \ref{primpropo}\\
$L10n89$ & $4$ & Prop \ref{primpropo}\\
$L10n90$ & $3$ & Prop \ref{primpropo}\\
$L10n91$ & $4$ & Prop \ref{primpropo}\\
$L10n92$ & $5$ & Prop \ref{primpropo}\\
$L10n93$ & $5$ & Prop \ref{primpropo}\\
$L10n94$ & $5$ & Prop \ref{primpropo}\\
$L10n95$ & $5$ & Prop \ref{primpropo}\\
$L10n96$ & $5$ & Prop \ref{primpropo}\\
$L10n97$ & $5$ & Prop \ref{primpropo}\\
$L10n98$ & $5$ & Prop \ref{primpropo}\\
$L10n99$ & $5$ & Prop \ref{primpropo}\\
$L10n100$ & $3$ & Prop \ref{primpropo}\\
$L10n101$ & $5$ & Prop \ref{primpropo}\\
$L10n102$ & $5$ & Prop \ref{primpropo}\\
$L10n103$ & $3$ & Prop \ref{primpropo}\\
$L10n104$ & $5$ & Prop \ref{primpropo}\\
$L10n105$ & $5$ & Prop \ref{primpropo}\\
$L10n106$ & $4$ & Prop \ref{primpropo}\\
$L10n107$ & $2$ & Prop \ref{primpropo}\\
$L10n108$ & $5$ & Prop \ref{primpropo}\\
$L10n109$ & $5$ & Prop \ref{primpropo}\\
$L10n110$ & $5$ & Prop \ref{primpropo}\\
$L10n111$ & $5$ & Prop \ref{primpropo}\\
$L10n112$ & $5$ & Prop \ref{primpropo}\\
$L10n113$ & $5$ & Prop \ref{primpropo}\\
\\
\bottomrule
\end{tabular}
\caption{Unlinking numbers of prime, non-split links with crossing number $10$}
\end{center}
\end{table}
\clearpage
\bibliography{bibliografia}
\bibliographystyle{plain}
\end{document}